\documentclass[oneside,10pt,reqno]{amsart}
\usepackage{amssymb,amsmath,amsthm,bbm,enumerate,mdwlist,url,multirow,hyperref,amsthm,amsaddr,mathrsfs}
\usepackage[pdftex]{graphicx}
\usepackage[shortlabels]{enumitem}
\usepackage{scalerel,stackengine}
\stackMath
\newcommand\reallywidehat[1]{%
\savestack{\tmpbox}{\stretchto{%
  \scaleto{%
    \scalerel*[\widthof{\ensuremath{#1}}]{\kern-.6pt\bigwedge\kern-.6pt}%
    {\rule[-\textheight/2]{1ex}{\textheight}}
  }{\textheight}%
}{0.5ex}}%
\stackon[1pt]{#1}{\tmpbox}%
}
\allowdisplaybreaks
\theoremstyle{definition}
\newtheorem{definition}{Definition}
\theoremstyle{theorem}
\newtheorem{proposition}[definition]{Proposition}
\newtheorem{lemma}[definition]{Lemma}
\newtheorem{theorem}[definition]{Theorem}

\numberwithin{equation}{section}
\numberwithin{definition}{section}
\theoremstyle{remark}
\newtheorem{example}[definition]{Example}
\def\lll{\mathsf{x}}

\def\QQ{\mathbb{Q}}
\def\PP{\mathbb{P}}

\def\MM{\mathcal{M}}

\def\ZZ{\mathbb{Z}}
 \def\supp{\mathrm{supp}}

\def\dd{\mathrm{d}}

\makeatletter
\@namedef{subjclassname@2020}{%
  \textup{2020} Mathematics Subject Classification}
\makeatother
\begin{document}
\title{Complete monotonicity of time-changed L\'evy processes  at first passage}

\author{Matija Vidmar}
\address{Department of Mathematics, Faculty of Mathematics and Physics, University of Ljubljana}
\email{matija.vidmar@fmf.uni-lj.si}

\begin{abstract}
We consider the class of (possibly killed) spectrally positive L\'evy process that have been time-changed by the inverse of an integral functional. Within this class  we characterize the family of those processes which satisfy the following property: as  functions of  point of issue, the Laplace transforms of their first-passage times downwards are completely monotone. A wide (dense, in a sense) subfamily of this family admits closed form expressions for said Laplace transforms.
\end{abstract}


\thanks{Support from the Slovenian Research Agency  (project No. N1-0174) is acknowledged.} 

\date{\today}

\maketitle

\section{Introduction}
As can be gathered from the literature, but is apparently not hitherto pointed out, (i) possibly killed spectrally positive L\'evy processes (pk-spLp) \cite{kyprianou}, (ii) continuous-state branching processes  \cite{ma}  and (iii) self-similar Markov processes of the spectrally positive type on the real line \cite{pierre} all share the following common property: as  functions of  point of issue, the Laplace transforms of their first-passage times downwards (fptd) are completely monotone (cm). More precisely, if $X$ is any of the processes mentioned, equipped with the probabilities $(\PP_x)_{x\in I}$ (where $I=\mathbb{R}$ for (i) and (iii), while $I=[0,\infty)$ or $I=(0,\infty)$ for (ii) according as to whether $0$ is hit with positive probability or not) in the usual way, then 
\begin{equation}\label{eq:cm-1}
\text{for each $q\in (0,\infty)$ and $\lll\in I$ the map }(0,\infty)\ni x\mapsto \PP_{\lll+x}[e^{-q T_\lll^-}]\text{ is  cm};
\end{equation}
here $T_\lll^-$ is the first entrance time of $X$ into $[\lll,\infty)$ and the reader is referred to \cite[Definition~1.3]{bernstein} for a definition of cm (the map must be of the class $C^\infty$ with derivatives alternating in sign). (Explicit expressions for the expectations appearing in \eqref{eq:cm-1} shall be recalled in Examples~\ref{example:levy-process}-\ref{example:logs-ppsmp}. Except for (i) they are quite non-trivial.) By Bernstein's \cite[Theorem~1.4]{bernstein}  \& Tonelli's theorems, and since the class of cm functions is a convex cone closed under pointwise limits \cite[Corollary~1.6]{bernstein}, \eqref{eq:cm-1} is actually equivalent to asking that for all $\lll\in I$ and all bounded cm $f:(0,\infty)\to \mathbb{R}$ vanishing at infinity the map $(0,\infty)\ni x\mapsto \PP_{\lll+x}[f( T_\lll^-)]$ is cm.

By the well-known Lamperti transforms \cite[Sections~12.1 \&~13.3]{kyprianou} in  cases (ii) \& (iii), trivially for (i), the classes (i)-(iii) are also all instances of time-changes of pk-spLp (by inverses of integral functionals). We limit ourselves to these and investigate when precisely the indicated property \eqref{eq:cm-1} of cm at first passage holds true for them.

Roughly speaking, our results are as follows. If $X$ is got from a pk-spLp by driving along its sample paths with a velocity that is the function $A$ of its position, then \eqref{eq:cm-1} is equivalent to $\frac{1}{A}$ being cm (Theorem~\ref{theorem:time-changed-cm}). In such case $\frac{1}{A}$ is the Laplace transform of a unique measure $\gamma$ on the Borel sets of $[0,\infty)$; under an extra condition, which is met if $\gamma$  has a support that is bounded away from zero,  we get  closed form Laplace transforms of the fptd (Proposition~\ref{corollary:P-sufficient}\ref{sufficient:iii}) that generalize those of (iii) above. This appears noteworthy as such explicit tractable expressions are generally hard to come by. And even if the extra condition does not prevail, knowing a priori the cm nature of the Laplace transforms of the fptd is a substantial structural property and can be a strong aid in evaluating them. Indeed, because cm functions are just Laplace transforms of  their ``representing'' measures, immediately one has an ansatz for them. Moreover, there is, in complete generality, a family of convolutional equations  for these representing measures (Eq.~\eqref{eq:nu-q}).

Time-changed (pk-)spLp have been studied quite extensively in the branching literature. See e.g. \cite{f-l-z,li-zhou}, where they are called nonlinear continuous-state branching processes. We make more concrete the relevance of \cite{f-l-z,li-zhou} to fptd in the comments following Definition~\ref{definition:measures-mq}. We might also mention here \cite[esp. Proposition~2.1]{LI20192941} for the special case of polynomial branching processes, where, among other results, the mean of the fptd \cite[Corollary~1.7]{LI20192941} is handled.\label{page:references}

Apart from the literature quoted above, a paper somewhat related to the preceding is \cite{temporally-cm} wherein it was considered which functions $f$ of  a Markov process $Y$ with state space $J$ and probabilities $(\QQ_y)_{y\in J}$ render the map $(0,\infty)\ni t\mapsto  \QQ_y[f(Y_t)]$  cm for all $y\in J$ -- ``temporal'' cm. The dual problem of so-called completely excessive measures (which is a notion related to cm \cite[Definitions~1.3 and~2.1]{saad}) was studied in \cite{saad}. Property~\eqref{eq:cm-1} ---  ``spatial'' cm at first passage --- can also be asked of a general $I$-valued Markov process $X$ (not necessarily time-changed spectrally positive L\'evy), $I$ an interval of $\mathbb{R}$ unbounded above, but establishing when it holds true then appears fundamentally more involved, and we do not pursue it here (though we may point out that positive examples outside the time-changed L\'evy world certainly exist,  e.g. such is the case for continuous-state branching processes with (im)migration \cite{ma,vidmar2021continuousstate}).  The reason for the restriction to processes without negative jumps (by taking the negative of course we subsume also those without positive jumps), besides the one given above, is that this property is responsible for a considerable simplification of the first passage theory, which is a well-known phenomenon \cite{kkr,vidmar2021exit,landriault_li_zhang_2017,avram_li_li_2021}; most directly this is visible in the existence of so-called scale functions (we introduce these for the present context in  Eq.~\eqref{eq:scale-identity} below). 

A final point, which we should like to make, is that the cm property \eqref{eq:cm-1} has its natural analogue in the discrete space setting, when the state space is an interval of $\ZZ$ unbounded above: cm becomes that of a sequence on $I$ \cite[Eq.~(VIII.3.1)]{feller}, but other than that, the same fundamental ideas apply (and analogous examples are available, e.g. continuous-time  Bienaym\'e-Galton-Watson processes \cite{Avram2019} correspond to (ii) above). Since the discrete space platform is technically less involved  and one expects that everything goes through, mutatis mutandis, we leave a parallel treatment thereof to the interested party.

\section{Preliminaries}
Fix $I$, an interval of $\mathbb{R}$, unbounded above, and put $I^\circ:=I\backslash \{\inf I\}$ for its interior.

\subsection{Complete monotonicity}\label{subsection:cm} A continuous real-valued map $\Theta$, defined on $I^\circ$ or $I$,  shall be said to be cm if $(0,\infty)\ni x\mapsto \Theta(\lll+x)$ is cm for all $\lll\in I$. When so, then  by Bernstein's theorem and since finite Laplace transforms (on a neighborhood  of infinity) determine measures on $[0,\infty)$ uniquely \cite[Theorem~8.4]{bhattacharya}, $\Theta$ is the Laplace transform of a unique (we will say, ``representing'') measure $\rho$ on $[0,\infty)$: $\Theta(\theta)=\hat{\rho}(\theta)$ for $\theta$ from the domain of $\Theta$. Here, and throughout, for a possibly signed measure $\rho$ on   $\mathcal{B}_{[0,\infty)}$ ($\mathcal{B}_{[0,\infty)}$ stands for the Borel sets of $[0,\infty)$) its Laplace transform $\hat{\rho}$ is given by $\hat{\rho}(\theta):=\int e^{-\theta z}\rho(\dd z)$ for those $\theta\in \mathbb{R}$ for which the integral $\int e^{-\theta z}\rho(\dd z)$ is well-defined (but possibly infinite).

\subsection{General notation}\label{subsectino:general-notation}
We shall write $g\cdot \mu:=(A\mapsto \int_Ag\dd\mu)$ (resp. $\mu[g]:=\int g\dd\mu$) for the indefinite (resp. definite) integral of a numerical measurable map $g$ against the measure $\mu$. On the other hand $g_\star\mu:=(A\mapsto \mu(g\in A))$ is the push-forward of a measure $\mu$ along a measurable $g$. And $\mu\star\nu:=+_\star(\mu\times \nu)$ is the convolution of the measures $\mu$ and $\nu$. As is customary, given an expression $f(x)$ for $x\in R$ we often write $f(\cdot)$ for the map $f$ defined on $R$, whose value at $x\in R$ is $f(x)$, $R$ being understood from context. For  two functions  $f_1$ and $f_2$ with the same domain and taking values in $[0,\infty]$, $f_1\propto f_2$ is taken to mean existence of $c\in (0,\infty)$ such that $f_2=cf_1$.

\subsection{Time-changed pk-spLp}\label{subsection:time-changed}
We construct the time-changed pk-spLp, which shall be the object of our study. Below, intuitively, while in $I^\circ$, the process $X$ should be viewed as having been got from the pk-spLp $\xi$ by driving along its sample paths with a velocity that is the function $A$ of its position; exiting from $I^\circ$ the process $X$ is stopped at $\inf I$ or relegated to the cemetery $\infty$ according as to whether  $X$ limits to $\inf I$ or not. 

Formally, let $A:I^\circ\to(0,\infty)$ be locally bounded, locally bounded away from zero and Borel measurable, let $\psi:[0,\infty)\to \mathbb{R}$ be the Laplace exponent of a spLp (for emphasis: no killing, $\psi(0)=0$), and let $p\in [0,\infty)$ (killing parameter, which we prefer to keep separate from $\psi$). For $\psi$ we exclude subordinators (in particular the constant process), but not negative drifts and assume further 
\begin{enumerate}[(a)]
\item\label{fund:a} in case $\inf I=-\infty$, that $A$ is bounded on $(-\infty,\lll]$ for some $\lll\in \mathbb{R}$;
\item\label{fund:b} in case $\inf I>-\infty$, that  $\inf I\in I$ iff $\int^\infty_{\psi^{-1}(0)+1} \frac{\dd\lambda}{\lambda A(\inf I+ \frac{1}{\lambda})\psi(\lambda)}<\infty$.
\end{enumerate}
These are our  deterministic input data. We will comment on the relevance of \ref{fund:a}-\ref{fund:b} below. 

We proceed to specify the associated stochastic objects. Let then $\xi=(\xi_u)_{u\in [0,\eta)}$ be a pk-spLp with Laplace exponent $\psi-p$  under the complete probabilities $(\PP_x)_{x\in \mathbb{R}}$ [thus $e^{u(\psi(\lambda)-p)}=\PP_x[e^{-\lambda(\xi_u-\xi_0)};u<\eta]$ for $\{u,\lambda\}\subset [0,\infty)$, $x\in \mathbb{R}$]; we insist that $\xi$ is c\`adl\`ag and has no negative jumps, both of these with certainty (not just a.s.). Set $$\sigma:=\sigma^-_{\inf I}:=\inf \{u\in [0,\eta):\xi_{u-}\land \xi_u\leq \inf I\}$$ [$\inf\emptyset:=\infty$, $\xi_{0-}:=\xi_0$ on $\{\eta>0\}$] and $$F(v):=\int_0^v \frac{\dd u}{A(\xi_u)}\text{ for }v\in [0,\eta\land\sigma].$$  Then, if $\inf I>-\infty$, we have \cite[Proposition~3.5]{vidmar2021characterizations} that $\inf I\in I$ iff 
\begin{equation}\label{eq:F-condition}
\text{$F(\sigma)<\infty$ with positive $\PP_x$-probability on $\{\sigma<\eta\}$ for some $x\in I^\circ$}
\end{equation}
(it is due to \ref{fund:b}: basically it follows from \cite[Theorem~2.1]{li2020integral}) in which case $F(\sigma)<\infty$ a.s.-$\PP_x$ [and also with positive $\PP_x$-probability] on $\{\sigma<\eta\}$ for all $x\in I^\circ$.  Put further $\tau:=F^{-1}$ on $[0,F(\eta\land\sigma))$ and define 
$$
\text{$\zeta:=F(\eta)\mathbbm{1}_{\{\sigma=\infty\}}+\infty\mathbbm{1}_{\{\sigma<\eta\}}$
 or $\zeta:=F(\eta\land \sigma)\overset{\text{a.s.}}{=}F(\eta)$ according as to whether $\inf I\in I$ or not,}$$
 and then finally, for $t\in [0,\zeta)$,
  $$X_t:=
\begin{cases}
\xi_{\tau_t}& \text{ if }t<F(\eta\land\sigma)\\
\inf I&\text{ if } t\geq F(\eta\land\sigma)
\end{cases}.
$$ 
 
We shall work with the process  $X=(X_t)_{t\in [0,\zeta)}$ under the probablities $(\PP_x)_{x\in  I}$. Its law is fully determined by the pair $(\psi-p,A)$. The assumptions made on this pair serve to ensure \cite[Proposition~3.5(i)]{vidmar2021characterizations} that $X$ is well-defined and well-behaved  in the sense of \cite[Subsection~1.1]{vidmar2021characterizations}. In particular \ref{fund:b} serves to guarantee, via \eqref{eq:F-condition}, that $X$ hits all levels  from $I$ below its starting point with positive probability, while \ref{fund:a} is used to ensure that $\lim_{\zeta-}X$ exists in $I\cup \{\infty\}$ a.s. on $\{0<\zeta<\infty\}$ (see proof of \cite[Proposition~3.5]{vidmar2021characterizations} for details). The classes (i), (ii) and (iii) from the Introduction correspond  to  $A\equiv 1$, $A$ the identity on $(0,\infty)$ or $[0,\infty)$ as the case may be and to  $A$ being an (increasing) exponential function, respectively.

\subsection{Scale functions for fptd of time-changed pk-spLp}
For $\lll\in I$, let  $$T_\lll^-:=\inf \{t\in [0,\zeta):X_t\leq \lll\}$$ be the first passage time of $X$ below the level $\lll$.

As already mentioned above, and established in \cite[Proposition~3.5(iii)]{vidmar2021characterizations},
\begin{equation}\label{eq:positive-chances}
\text{$\PP_x(T_\lll^-<\zeta)>0$ for all $\lll\leq x$ from $I$};
\end{equation}
therefore \cite[Eq.~($q$) in Subsection~1.2]{vidmar2021characterizations}, for each $q\in [0,\infty)$, there exists a, unique up to a  multiplicative constant from $(0,\infty)$,  so-called scale function $\Phi_q:I\to (0,\infty)$ such that 
\begin{equation}\label{eq:scale-identity}
\PP_x[e^{-q T_\lll^-};T_\lll^-<\zeta]=\frac{\Phi_q(x)}{\Phi_q(\lll)}\text{ for all $\lll\leq x$ from $I$}.
\end{equation}
By  \cite[Proposition~3.5(iii)]{vidmar2021characterizations} $\Phi_0\propto e^{-\psi^{-1}(p)\cdot }$, where $\psi^{-1}:[0,\infty)\to [0,\infty)$ is the right-continuous inverse of $\psi$: $$\psi^{-1}(u):=\inf\{s\in [0,\infty):\psi(s)> u\},\quad  u\in [0,\infty).$$

Put next ($E^F$ stand for the set of functions mapping a set $F$ into a set $E$) $$\MM_{I}:=\left\{\nu\in[0,\infty]^{\mathcal{B}_{[0,\infty)}}:\nu\text{ a measure on $\mathcal{B}_{[0,\infty)}$ and }\hat{\nu}(\lll)<\infty\text{ for all }\lll\in I\right\}.$$ Because the scale functions are automatically continuous \cite[Theorem~2.2]{vidmar2021characterizations} we see that \eqref{eq:cm-1}, which by definition is the cm of the $\Phi_q$, $q\in (0,\infty)$, is equivalent to: 
\begin{equation}\label{cm-equiv}
\forall q\in (0,\infty)\, \exists \nu_q\in \MM_I\text{ such that } \Phi_q\propto \widehat{\nu_q}\vert_I,
\end{equation}
in which case the measures $\nu_q$, $q\in (0,\infty)$, are unique up to a multiplicative constant from $(0,\infty)$, they are non-zero and locally finite. By \cite[Propositions~3.1,~3.4 \&~3.5(ii)]{vidmar2021characterizations}, for a given $\nu\in \MM_I\backslash\{0\}$ (here $0:=(\mathcal{B}_{[0,\infty)}\ni E\mapsto 0)$ is the degenerate everywhere zero measure on  $\mathcal{B}_{[0,\infty)}$) and $q\in (0,\infty)$, 
 \begin{equation}\label{scale-gen}
\Phi_q\propto \hat{\nu}\vert_I\Leftrightarrow \left( A\reallywidehat{(\psi-p)\cdot \nu}=q\hat{\nu}\text{ on } I^\circ \text{ and ($\{0<\zeta<\infty,X_{\zeta-}=\infty\}$ is negligible or $\nu(\{0\})=0$)}\right).
 \end{equation}
The condition that $\{0<\zeta<\infty,X_{\zeta-}=\infty\}$ be negligible is that a.s. explosion does not occur; this is automatic if $p>0$ or else $p=0=\psi^{-1}(0)$ (in the former case $\xi$ is killed a.s. before $X$ can explode, in the latter $\xi$ does not drift to $\infty$ and again explosion cannot take place). It makes sense to, and we do specify $\nu_0$ as being the measure $\delta_{\psi^{-1}(p)}$ (up to a multiplicative constant from $(0,\infty)$), so that $\Phi_0\propto \widehat{\nu_0}\vert_I$. 
 
Lastly, set $\psi^\#:=\psi(\psi^{-1}(p)+\cdot)-p$, so that $\psi^\#$ is the Laplace exponent of a (for emphasis: not killed) spLp that is not drifting to $\infty$ (in terms of $\psi^\#$ it means that $\psi^\#(0)=0$ and $(\psi^\#)^{-1}(0)=0$) -- the Esscher transform. With, ceteris paribus, $\psi^\#$ the Laplace exponent of $\xi$,  i.e. with $(\psi^\#,A)$ the input data for $X$, we get $X^\#$ under the probabilities $(\PP^\#_x)_{x\in I}$ etc. According to \cite[Proposition~3.5(iv)]{vidmar2021characterizations}, for each $q\in [0,\infty)$, 
\begin{equation}\label{eq:measure-change}
\Phi_q\propto \Phi_0\times\Phi_q^\#\propto e^{-\psi^{-1}(p)\cdot}\times \Phi_q^\#.
\end{equation}

\section{The characterization and related results}\label{setions:conditions}
We prepare a reduction technique of ``exponential tilting'', which shall simplify the arguments and especially the notational overhang.

\begin{lemma}\label{lemma:measure-change}
\eqref{eq:cm-1} holds true for $X$ under $(\PP_x)_{x\in I}$ iff it holds true for $X^\#$ under $(\PP^\#_x)_{x\in I}$, in which case $\nu_q\propto (\cdot +\psi^{-1}(p))_\star \nu_q^\#$ (i.e. $\nu_q(\psi^{-1}(p)+\dd y)\propto \nu_q^\#(\dd y)$ in $y\in [0,\infty)$ and $\nu_q([0,\psi^{-1}(p)))=0$). In particular, if \eqref{eq:cm-1} is verified, then necessarily $\nu_q$ is carried by $[\psi^{-1}(p),\infty)$ for all $q\in (0,\infty)$.
\end{lemma}
Basically it means that in the proofs to follow we will be able to limit ourselves, without loss of generality, to the case when $\psi^{-1}(0)=0=p$.
\begin{proof}
Immediate from \eqref{cm-equiv} and \eqref{eq:measure-change}.
\end{proof}
The characterization of \eqref{eq:cm-1} we split into two legs. Here is the first.

\begin{proposition}\label{proposition:P-necessary}
If \eqref{eq:cm-1} holds true, then $\frac{1}{A}$ is cm.
\end{proposition}
\begin{proof}
By Lemma~\ref{lemma:measure-change} we may and do assume $\psi^{-1}(0)=0=p$. Fix $\lll\in I^\circ$. By \eqref{cm-equiv}-\eqref{scale-gen}
$$A  \reallywidehat{\frac{\psi-p}{q}\cdot\nu_q}=\widehat{\nu_q}\text{ on }[\lll,\infty),\quad q\in (0,\infty),$$
where we may and do assume (by renormalizing the measures if necessary) that $\widehat\nu_q(\lll)=1$ for all $q\in (0,\infty)$. Now $\widehat{\nu_q}=\PP_\cdot[e^{-qT_\lll^-};T_\lll^-<\zeta]\uparrow \PP_\cdot(T_\lll^-<\zeta)= 1$ on $[\lll,\infty)$ as $q\downarrow 0$, since $\psi^{-1}(0)=0=p$. By the continuity theorem for Laplace transforms \cite[Theorem~XIII.1.2a]{feller} it follows that $\reallywidehat{\frac{\psi-p}{q}\cdot\nu_q}$ is converging pointwise on $[\lll,\infty)$ to $\hat\rho$ for some measure $\rho$ on $\mathcal{B}_{[0,\infty)}$ as $q\downarrow 0$. Thus $$A\hat{\rho}=1\text{, i.e. }\frac{1}{A}=\hat\rho,$$ both on $[\lll,\infty)$. Since this is true of every $\lll\in I^\circ$ it follows that $\frac{1}{A}$ is in fact cm. 
\end{proof}
For the converse, second leg, we shall find it convenient to introduce
\begin{definition}\label{definition:measures-mq}
Suppose $\frac{1}{A}$ is cm. Let $\gamma\in \MM_{I^\circ}\backslash \{0\}$ be the associated representing measure  for which $A=\frac{1}{\hat{\gamma}}$ on $I^\circ$. For $q\in (0,\infty)$ we define
\begin{align}
\nonumber m_q&:=\sum_{k\in \mathbb{N}_0} q^k\underbrace{\frac{1}{\psi-p}\cdot\left(  \left(\cdots\frac{1}{\psi-p}\cdot\left(\left(\frac{1}{\psi-p}\cdot \left(\delta_{\psi^{-1}(p)}\star \gamma\right)\right)\star\gamma\right)\cdots \right)\star\gamma\right)}_{\text{$k$-times}}\\
&=\delta_{\psi^{-1}(p)}+\frac{q}{\psi-p}\cdot \left((\psi^{-1}(p)+\cdot)_\star\gamma\right)+\frac{q}{\psi-p}\cdot\left(\left(\frac{q}{\psi-p}\cdot (\psi^{-1}(p)+\cdot)_\star\gamma\right)\star \gamma\right)+\cdots,\label{eq:new-measures}
\end{align}
(a measure on $\mathcal{B}_{[0,\infty)}$) and we interpret $\frac{q}{0}=\infty$.
\end{definition}
 In a very camouflaged way the measures of \eqref{eq:new-measures} make their appearence in  \cite[Theorem~3.1(ii)]{vidmar_2019} and again in \cite[Theorem~4.3]{f-l-z}, their applicability being subject to  sufficient conditions that we refine here to an equivalence.  They are all equivalent as $q\in (0,\infty)$ varies and satisfy the convolutional relations 
 \begin{equation}\label{eq:convolutional-relation}
 m_q=\delta_{\psi^{-1}(p)}+\frac{q}{\psi-p}\cdot (m_q\star \gamma),\quad q\in (0,\infty),
 \end{equation} 
 which can be found  in \cite[Remark~4.6]{li-zhou} for the case when $\psi^{-1}(0)=0=p$. The simplified form of $m_q$, $q\in (0,\infty)$, which results when $\psi^{-1}(0)=0=p$ is worth highlighting: 
 \begin{equation}\label{eq:simplified}
m_q=\sum_{k\in \mathbb{N}_0} q^k\frac{1}{\psi}\cdot\left(  \left(\cdots\frac{1}{\psi}\cdot\left(\left(\frac{1}{\psi}\cdot \left(\delta_0\star \gamma\right)\right)\star\gamma\right)\cdots \right)\star\gamma\right)=\delta_0+\frac{q}{\psi}\cdot \gamma+\frac{q}{\psi}\cdot\left(\left(\frac{q}{\psi}\cdot \gamma\right)\star\gamma\right)+\cdots.
\end{equation}

\begin{proposition}\label{corollary:P-sufficient}
Assume  $\frac{1}{A}$ is cm. Let $\gamma\in \MM_{I^\circ}\backslash \{0\}$ be the associated representing measure  for which $A=\frac{1}{\hat{\gamma}}$ on $I^\circ$. Denote $\alpha:=\inf\mathrm{supp}(\gamma)$.

\begin{enumerate}[(i)]
\item\label{sufficient:i} We have that  \eqref{eq:cm-1} holds true, and therefore \eqref{cm-equiv} is also true, which gives us access to the $\nu_q$, $q\in (0,\infty)$, each of which is unique up to a multiplicative constant from $(0,\infty)$.
\item\label{sufficient:ii}  For each $q\in (0,\infty)$,  $\nu_q$ is identified, uniquely up to a multiplicative constant from $(0,\infty)$, by the conditions: $\nu_q\in \MM_I\backslash\{0\}$,  $\nu_q$ is carried by $[\psi^{-1}(p),\infty)$ and  
\begin{equation}\label{eq:nu-q}
(\psi-p)\cdot \nu_q=q\gamma\star\nu_q.
\end{equation}
\item\label{sufficient:iii} For all $q\in (0,\infty)$, 
$$\alpha>0\Rightarrow \nu_q(\{\psi^{-1}(p)\})>0\Leftrightarrow   \nu_q\propto m_q \Leftrightarrow m_q\in \MM_I\Rightarrow \gamma(\{0\})=0.$$
\end{enumerate}
  \end{proposition}
 \begin{proof}
By Lemma~\ref{lemma:measure-change} we may and do assume that $\psi^{-1}(0)=0=p$.

 \ref{sufficient:ii}. Once \ref{sufficient:i} has been established  this is just  \eqref{scale-gen}, since the event  $\{0<\zeta<\infty,X_{\zeta-}=\infty\}$ (of explosion) is negligible thanks to $\psi^{-1}(0)=0$. But anyway already now we know that for a given $q\in (0,\infty)$ and  $\nu\in \MM_I\backslash\{0\}$,  $\Phi_q\propto\hat{\nu}\vert_I$ iff  $(\psi-p)\cdot \nu=q\gamma\star\nu$.
 
\ref{sufficient:i}. Suppose  $\alpha>0$ and $\gamma\in \MM_I$ in the first instance. Let $q\in (0,\infty)$. From \eqref{eq:simplified} we estimate $$\widehat{m_q}\leq \sum_{k\in \mathbb{N}_0}\frac{(q\hat{\gamma})^k}{\prod_{l=1}^k\psi(\alpha l)}<\infty\text{ on }I, \quad q\in (0,\infty),$$ where we have also used the fact that $\psi$ grows at least linearly at $\infty$. Thus $m_q\in\MM_I\backslash \{0\}$. Besides, by \eqref{eq:convolutional-relation}, 
$\psi\cdot m_q=qm_q\star \gamma$. We see that  we indeed have \eqref{eq:cm-1}, therefore \eqref{cm-equiv}, and $\nu_q\propto m_q$ for all $q\in (0,\infty)$.

The general case is handled by taking pathwise limits, approximating $A$. For $\epsilon\in (0,\infty)$ let $\gamma^\epsilon:=\gamma(\{0\})\delta_\epsilon +\mathbbm{1}_{[\epsilon,1/\epsilon)}\cdot \gamma$ or  $\gamma^\epsilon:=\gamma(\{0\})\delta_\epsilon +\mathbbm{1}_{[\epsilon,\infty)}\cdot \gamma$ according as to whether $\inf I\in I$ or $\inf I\notin I$. The measure $\gamma^\epsilon$ is non-zero for all sufficiently small $\epsilon>0$, we restrict to those and note that
\begin{itemize}
\item $\gamma^\epsilon$ has support bounded away from zero and belongs to $\MM_I$; 
\item $A^\epsilon:=\frac{1}{\widehat{\gamma^\epsilon}\vert_{I^\circ}}$ meets the conditions of Subsection~\ref{subsection:time-changed} in lieu of $A$;
\item as $\epsilon\downarrow 0$, $\frac{1}{A^\epsilon}=\widehat{\gamma^\epsilon}\to \hat{\gamma}=\frac{1}{A}$ pointwise and  boundedly on subsets of $I^\circ$ bounded from below. 
 \end{itemize}
 Let $q\in (0,\infty)$ and $\lll\in I$. By what we have just shown, in the obvious notation, for the indicated small enough $\epsilon>0$,
\begin{equation}\label{eq:approx}
(0,\infty)\ni x\mapsto \PP_{\lll +x}[e^{-q T_\lll^{\epsilon-}};T_\lll^{\epsilon-}<\zeta^\epsilon]=\frac{\widehat{m^\epsilon_q}(\lll+x)}{\widehat{m^\epsilon_q}(\lll)}
\end{equation}
 is cm.  If we let $\sigma_\lll^-:=\inf\{u\in [0,\eta):\xi_u\land \xi_{u-}\leq \lll\}$, then, by bounded convergence, a.s. $T_\lll^{\epsilon-}=\int_0^{\sigma_\lll^-}\frac{\dd u}{A^\epsilon(\xi_u)}\to  \int_0^{\sigma_\lll^-}\frac{\dd u}{A(\xi_u)}=T_\lll^-$ as $\epsilon\downarrow 0$ on $\{\sigma_\lll^-<\eta\}=\{T_\lll^{\epsilon-}<\zeta^\epsilon\}=\{T_\lll^-<\zeta\}$. Thus, passing to the limit $\epsilon\downarrow 0$ in \eqref{eq:approx} by bounded convergence again, we get \eqref{eq:cm-1}. Furthermore, if $m_q\in \MM_I$, then clearly we must have $\gamma(\{0\})=0$ so that by monotone convergence $\widehat{m^\epsilon_q}\uparrow \widehat{m_q}$ on $I$ as  $\epsilon\downarrow 0$, which in view of \eqref{eq:approx} means that $\nu_q\propto m_q$.

\ref{sufficient:iii}.  We have already seen in the proof of \ref{sufficient:i} that the condition $m_q\in \MM_I$ is sufficent for $\nu_q\propto m_q$, as well as for $\gamma(\{0\})=0$. In order for $\nu_q\propto m_q$ it must be that $\nu_q(\{0\})>0$, just because $m_q(\{0\})\geq 1>0$. This establishes the $\Leftarrow$ directions of the two equivalences as well as the last $\Rightarrow$ implication.

Suppose now that $\nu_q(\{0\})>0$, without loss of generality $=1$. Put $\zeta_q:=\nu_q-\delta_0$. From \eqref{eq:nu-q} we get $\psi\cdot \zeta_q= q\gamma+q\gamma\star\zeta_q$. Thus $\gamma(\{0\})=0$ and $\nu_q= \delta_0+\frac{q}{\psi}\cdot \gamma+\frac{q}{\psi}\cdot (\zeta_q\star \gamma)$. Another insertion of $\zeta_q=\frac{q}{\psi}\cdot (\gamma+\gamma\star\zeta_q)$ gives $\nu_q= \delta_0+\frac{q}{\psi}\cdot \gamma+\frac{q}{\psi}\cdot\left(\left(\frac{q}{\psi}\cdot \gamma\right)\star\gamma\right)+\frac{q}{\psi}\cdot\left(\left(\frac{q}{\psi}\cdot \left(\zeta_q\star \gamma\right)\right)\star\gamma\right)$. Inductively we conclude that $\nu_q\geq m_q$. But $\nu_q\in \MM_I$, therefore $m_q\in \MM_I$. This concludes establishing the two equivalences.

Finally, suppose that $\alpha>0$; we check the first $\Rightarrow$ implication by verifying that $m_q\in \MM_I$. Let again $q\in (0,\infty)$. We have already verified in the proof of \ref{sufficient:i} that $m_q\in \MM_I$ if in addition we assumed $\gamma\in \MM_I$. Since anyway automatically $\gamma\in \MM_{I^\circ}$ all that is left to consider is the case when $\inf I\in I$. But in that case, for any given $l\in I^\circ$ we may apply the result to the interval $[l,\infty)$ in lieu of $I$ to get 
 $$\frac{\Phi_q(x)}{\Phi_q(\lll)}=\PP_x[e^{-q T_\lll^-};T_\lll^-<\zeta]=\frac{\widehat{m_q}(x)}{\widehat{m_q}(\lll)}\text{ for all $\lll\leq x$ from $[l,\infty)$, $\therefore$ for all $\lll\leq x$ from $I^\circ$}.$$
In the preceding we may now pass to the limit $\lll\downarrow\inf I$ for any fixed $x\in I^\circ$ and get, due to  monotone convergence and the continuity of the scale functions, that necessarily $\widehat{m_q}(\inf I)<\infty$, i.e. $m_q\in \MM_I$.

Thus all the implications (and equivalences) have been verified.
\end{proof}
Combining Propositions~\ref{proposition:P-necessary} and~\ref{corollary:P-sufficient} we get

\begin{theorem}\label{theorem:time-changed-cm}
The following statements are equivalent.
\begin{enumerate}[(I)]
\item Condition~\eqref{eq:cm-1} holds true.
\item $\frac{1}{A}$ is cm.\qed
\end{enumerate}
\end{theorem}
Let us see explicitly  how the cases (i)-(iii) from the Introduction are reflected in these results. 

\begin{example}\label{example:levy-process}
$I=\mathbb{R}$ and $A\equiv 1=\frac{1}{\widehat{\delta_0}}$. Then $X=\xi$ is just a pk-spLp under the probabilities $(\PP_x)_{x\in \mathbb{R}}$, Laplace exponent $\psi-p$. For $q\in (0,\infty)$,  $\nu_q\propto \delta_{\psi^{-1}(p+q)}$ \cite[p.~232, 1st display]{kyprianou}. In this case $m_q=\infty\delta_{\psi^{-1}(p)}\notin \MM_I$ and the indicated $\nu_q$ is indeed the (up to a  multiplicative constant from $(0,\infty)$)   only solution to \eqref{eq:nu-q} within $\MM_I\backslash \{0\}$ carried by $[\psi^{-1}(p),\infty)$, as it should be.
\end{example}
In an informal sense we can interpret Theorem~\ref{theorem:time-changed-cm} as follows: the cm of $\frac{1}{A}$ is precisely what is required for the cm property at first passage  \eqref{eq:cm-1} [equiv. \eqref{cm-equiv}] to be retained when passing from a pk-spLp (being then true in a very simple sense, with Dirac $\nu_q$, $q\in (0,\infty)$) to  the process which we get by driving along the sample paths of the pk-spLp with velocity $A$.

\begin{example}\label{example:csbp}
 $I=(0,\infty)$ or $I=[0,\infty)$ according as to whether $\int^\infty_{\psi^{-1}(0)+1} \frac{1}{\psi}=\infty$  or $\int^\infty_{\psi^{-1}(0)+1} \frac{1}{\psi}<\infty$, and $A=\mathrm{id}_{(0,\infty)}=\frac{1}{\widehat{\mathscr{L}}\,\vert_{(0,\infty)}}$, where $\mathscr{L}$ is Lebesgue measure on $\mathcal{B}_{[0,\infty)}$. Then $X$ is a continuous-state branching process  under the probabilities $(\PP_x)_{x\in I}$, branching mechanism $\psi-p$; and by the  Lamperti transform continuous-state branching processes, whose paths are not a.s. nondecreasing, are actually exhausted (in law) by this construction  \cite[Section~12.1]{kyprianou}. For $q\in (0,\infty)$,  $\nu_q\propto \left(\mathbbm{1}_{(\psi^{-1}(p),\infty)}\frac{1}{\psi-p}\exp\left(\int^{\cdot}_\theta\frac{q}{\psi-p}\right)\right)\cdot \mathscr{L}$, where  $\theta\in (\psi^{-1}(p),\infty)$ is arbitrary \cite[Theorem~1]{ma}. Again  $m_q\notin \MM_I$ ($\therefore$  $\nu_q(\{\psi^{-1}(p)\})=0$) and  the indicated $\nu_q$ is  the only solution to \eqref{eq:nu-q} within $\MM_I\backslash \{0\}$ carried by $(\psi^{-1}(p),\infty)$: the former because of the at least linear decay of $\psi-p$ at $\psi^{-1}(p)+$; the latter because $(\psi-p)\cdot \nu_q=q\mathscr{L}\star\nu_q$, together with $\nu_q$ being carried by $(\psi^{-1}(p),\infty)$, implies that $\nu_q\ll\mathscr{L}$, so that writing $w_q:=\frac{\dd \nu_q}{\dd\mathscr{L}}$, we get $(\psi-p)w_q=q\int_0^\cdot w_q$ a.e.-$\mathscr{L}$, which is solved for in a straightforward manner.
\end{example}

Generalizing the preceding, if the representing measure $\gamma$ of Proposition~\ref{corollary:P-sufficient} is absolutely continuous (w.r.t. Lebesgue measure $\mathscr{L}$), then, for each $q\in (0,\infty)$,  by  \eqref{eq:nu-q} and the translation invariance of Lebesgue measure, $(\psi-p)\cdot \nu_q$ is absolutely continuous also. So, by Proposition~\ref{corollary:P-sufficient}\ref{sufficient:iii}, either (a) $m_q\in \MM_I$ and $\nu_q\propto m_q$, or else (b) $\nu_q$ is carried by $(\psi^{-1}(p),\infty)$, therefore is itself absolutely continuous. If further $\zeta:=\frac{\dd\gamma}{\dd \mathscr{L}}$ is bounded away from zero a.e.-$\mathscr{L}$ locally at $0+$, then automatically (a) is precluded and $w_q:=\frac{\dd \nu_q}{\dd\mathscr{L}}$ satisfies $$(\psi-p)w_q=q\zeta\star w_q\text{ a.e.-$\mathscr{L}$},$$
where $\star$ is now just the usual convolution of functions on $[0,\infty)$.

\begin{example}\label{example:logs-ppsmp}
Fix $\alpha\in (0,\infty)$. $I=\mathbb{R}$ and $A=e^{\alpha \cdot}=\frac{1}{\widehat{\delta_\alpha}}$. Then $\exp(-X)$ under the probabilities $\QQ_y:=\PP_{-\log y}$, $y\in (0,\infty)$, is the  positive self-similar Markov process of the spectrally negative type associated to $-\xi$ and the index $\alpha$ via the (another) Lamperti transform (we view $0$ as a cemetery state for $\exp(-X)$); and positive self-similar Markov processes with $0$ absorbing and no positive jumps of index $\alpha$, whose paths are not a.s. nonincreasing, are actually exhausted (in law) by this construction \cite[Section~13.3]{kyprianou}.  For $q\in (0,\infty)$,   $\nu_q\propto m_q=\sum_{k=0}^\infty\frac{q^k}{\prod_{l=1}^k\psi(\psi^{-1}(p)+l\alpha)-p}\delta_{\psi^{-1}(p)+\alpha k}$  \cite[Theorem~13.10(ii)]{kyprianou}. On this example, let us check that the condition that $\nu_q$ be carried by $[\psi^{-1}(p),\infty)$ in Proposition~\ref{corollary:P-sufficient}\ref{sufficient:i} cannot be suspended. Indeed, if for instance $\alpha>\psi^{-1}(0)>0=p$, then $m_q':=\sum_{k=0}^\infty \frac{q^k}{\prod_{l=1}^k\psi(l\alpha)}\delta_{\alpha k}$, like $m_q$, solves \eqref{eq:nu-q} for $\nu_q$ and belongs to $\MM_I\backslash \{0\}$, but clearly $m_q'\not\propto m_q$. (It can only mean that under the preceding provisos explosion occurs with positive probability, as it does \cite[Theorem~13.1(i)(2)]{kyprianou}.)
\end{example}
An easy extension of this last example is when the $\gamma$ of Proposition~\ref{corollary:P-sufficient} is carried by the lattice $\alpha\mathbb{N}$ for some $\alpha\in (0,\infty)$. For in such case, given any $q\in (0,\infty)$, the measure $\nu_q$ is carried by $\psi^{-1}(p)+\alpha\mathbb{N}_0$ and expresses as  (employing the notation $[k]:=\{1,\ldots,k\}$ for $k\in \mathbb{N}_0$) $$\nu_q\propto m_q=\sum_{k\in \mathbb{N}_0}\sum_{n\in \mathbb{N}^{[k]}}q^k \left(\prod_{i\in [k]}\frac{\gamma(\{\alpha n_i\})}{\psi(\psi^{-1}(p)+\alpha(n_1+\cdots+n_i))-p}\right)\delta_{\psi^{-1}(p)+\alpha \sum n}.$$

As we have just beared witness to, the measures $\nu_q$, $q\in (0,\infty)$, are Dirac, absolutely continuous, discrete (even lattice, but not Dirac) in each of the respective instances (i)-(iii) from the Introduction. Proposition~\ref{corollary:P-sufficient}\ref{sufficient:iii} and the explicit form \eqref{eq:new-measures} show that $\nu_q$, $q\in (0,\infty)$, having both discrete and absolutely continuous, even singular components all at once are possible. 

For our final result still $I$ is fixed, but two processes of the same type as $X$ are considered. We assume that their cm scale functions \eqref{cm-equiv} agree up to multiplicative constants from $(0,\infty)$ and that for each of them $\inf\supp(\gamma)>0$, where $\gamma$ is the representing measure of $\frac{1}{A}$,  and are able to conclude --- leaning heavily on the explicit form  \eqref{eq:new-measures} --- that they have the same law. It is a (relatively weak, but nevertheless) complement to the characterization results of \cite{vidmar2021characterizations}.

\begin{proposition}\label{proposition:equality-in-law}
Let $\{\gamma^1,\gamma^2\}\subset \MM_{I^\circ}\backslash \{0\}$, $\{p^1,p^2\}\subset [0,\infty)$ and let $\psi^1$, $\psi^2$ be two Laplace exponents of spLp. For $i\in \{1,2\}$ let $X^i$ under the probabilities $(\PP_x^i)_{x\in I}$  be associated to $(\psi^i-p^i,A^i:=\frac{1}{\widehat{\gamma^i}})$ as $X$ under $(\PP_x)_{x\in I}$ is to $(\psi-p,A)$. The pairs $(\psi^i-p^i,A^i)$, $i\in \{1,2\}$, are of course assumed to satisfy the standing assumptions of Subsection~\ref{subsection:time-changed} in lieu of $(\psi-p,A)$ and we insist further that $\alpha^i:=\inf\supp(\gamma^i)>0$ for $i\in \{1,2\}$. Suppose the processes $X^1$ and $X^2$ have the same laws of their fptd, i.e., in the obvious notation, ${T_\lll^{1-}}_\star \PP_x^1={T_\lll^{2-}}_\star \PP_x^2$ for all $\lll\leq x$ from $I$. Then $X^1$ has the same law as $X^2$, that is to say $\PP^1_x=\PP^2_x$ for all $x\in I$.
\end{proposition}
\begin{proof}
Equality of the laws of the fptd implies equality of their Laplace transforms, which in turn renders that the associated scale functions and therefore their representing measures agree up to   multiplicative constants from $(0,\infty)$. 

Since $\delta_{(\psi^1)^{-1}(p^1)}\propto \nu_0^1\propto \nu_0^2\propto \delta_{(\psi^2)^{-1}(p^2)}$, it must be that $(\psi^1)^{-1}(p^1)=(\psi^2)^{-1}(p^2)$. Therefore, as Laplace exponents of pk-spLp (being analytic on $\{\Re>0\}$ and continuous on $\{\Re\geq 0\}$) are determined by their values on a set with an accumulation point, in particular by their values on a (real) neighborhood of $\infty$, and taking into account also  Lemma~\ref{lemma:measure-change}, we see that we may, and we  do assume  $p^1=0=p^2$ and $(\psi^1)^{-1}(0)=0=(\psi^2)^{-1}(0)$.

On account of \eqref{scale-gen}, $\widehat{\psi ^i\cdot \nu_q^i}=q\widehat{\nu^i_q}\widehat{\gamma^i}$ on $I^\circ$ for $q\in (0,\infty)$ and $i\in \{1,2\}$; therefore it will suffice to establish that $\psi^1\propto \psi^2$.

By Proposition~\ref{corollary:P-sufficient}\ref{sufficient:iii}, for all $q\in (0,\infty)$, $m^1_q=m^2_q$; inspecting the $m_q^i$, $i\in \{1,2\}$,  as power series in $q\in (0,\infty)$, we  deduce that the measures $\gamma_k^i$, $i\in\{1,2\}$, at each of the $q^k$, $k\in \mathbb{N}$, in \eqref{eq:simplified} agree (at $k=0$ also, but trivially, just $\delta_0$ for both $X^1$ and $X^2$). Thus
\begin{align}
\frac{1}{\psi^1}\cdot \gamma^1&=\frac{1}{\psi^2}\cdot \gamma^2\label{relation:two}\\
\frac{1}{\psi^1}\cdot\left(\left(\frac{1}{\psi^1}\cdot \gamma^1\right)\star\gamma^1\right)&=\frac{1}{\psi^2}\cdot\left(\left(\frac{1}{\psi^2}\cdot \gamma^2\right)\star\gamma^2\right)\label{relation:one}\\
&\text{etc.}\nonumber
\end{align}
From \eqref{relation:two} and the fact that $\psi^1$ and $\psi^2$ are $(0,\infty)$-valued on $(0,\infty)$ we get that $$\alpha^1=\inf\supp\left(\frac{1}{\psi^1}\cdot \gamma^1\right)=\inf\supp\left(\frac{1}{\psi^2}\cdot \gamma^2\right)=\alpha^2$$ and we write just $\alpha$ for the common value;  evaluating  then \eqref{relation:two} at $[\alpha,\alpha+\epsilon]$ we have further that $$\frac{\gamma^1([\alpha,\alpha+\epsilon])}{\psi^1(\alpha)}\sim \frac{\gamma^2([\alpha,\alpha+\epsilon])}{\psi^2(\alpha)}\text{ as }\epsilon\downarrow 0,$$
where we have also used the local Lipschitz continuity of $\frac{1}{\psi^1}$ and $\frac{1}{\psi^2}$ on $(0,\infty)$. By the very same token we get from \eqref{relation:one}  that
$$\psi^1(\alpha)\left[\frac{1}{\psi^1}\cdot\left(\left(\frac{1}{\psi^1}\cdot \gamma^1\right)\star\left(\frac{1}{\psi^1}\cdot\gamma^1\right)\right)\right]([2\alpha,2\alpha+\epsilon])\sim\psi^2(\alpha)\left[\frac{1}{\psi^2}\cdot\left(\left(\frac{1}{\psi^2}\cdot \gamma^2\right)\star\left(\frac{1}{\psi^2}\cdot\gamma^2\right)\right)\right]([2\alpha,2\alpha+\epsilon])$$  as $\epsilon \downarrow 0$, and therefore, applying the same kind of argument yet again that $$\frac{\psi^1(\alpha)}{\psi^1(2\alpha)}=\frac{\psi^2(\alpha)}{\psi^2(2\alpha)}.$$ Continuing inductively in this manner it results that $$\frac{\psi^1(\alpha)}{\psi^1(k\alpha)}=\frac{\psi^2(\alpha)}{\psi^2(k\alpha)},\quad k \in \mathbb{N}_{\geq 2}.$$ In other words $\psi^1\vert_{\alpha\mathbb{N}}\propto \psi^2\vert_{\alpha\mathbb{N}}$. Since \cite[Lemma~4.5]{vidmar2021characterizations} the Laplace exponent of a spLp not drifting to infinity is determined by its values on  a given lattice $\alpha\mathbb{N}$, we deduce that $\psi^1\propto \psi^2$, which concludes the proof.
\end{proof}

\section{Conclusion}
In closing, let us briefly indicate some further   --- besides the extension to more general Markov processes and  the discrete space  complement that were already mentioned in the Introduction --- problems associated to, but left untreated in the above, and that may  be interesting for future study. One is the efficient numerical evaluation of the measures of Definition~\ref{definition:measures-mq} or indeed of their Laplace transforms. Another relates to Proposition~\ref{corollary:P-sufficient}: in \ref{sufficient:iii} thereof we identify $\nu_q$ explicitly (as $\propto m_q$) when $m_q\in \MM_I$. Can the $\nu_q$ be identitified in some relatively explicit way also when $m_q\notin \MM_I$ (at least under some reasonable sufficient condition that is given directly in terms of $(\psi-p,A)$)? The final  comes from the observation --- we do not make it completely precise --- that, at least judging by the self-similar \cite[Eq.~(3.1)]{vidmar_2020} and branching \cite[Eq.~(3.7)]{vidmar2021continuousstate} cases, some cm appears to be present also in the evaluation of the Laplace transforms of the lifetime $\zeta$ of $X$,  albeit in a less clear-cut way. Thus it may well be worth exploring whether here too it is the cm of $\frac{1}{A}$ which is, as it were, responsible for this phenomenon.

%
%
%
%

\bibliographystyle{plain}
\bibliography{Branching}

\end{document}